\documentclass[10pt,reqno]{amsart}

\usepackage{amsmath,amssymb,amsfonts,latexsym,amsthm}

\usepackage{mathrsfs}

\numberwithin{equation}{section}

\usepackage{graphicx}

\usepackage{ifthen} 

\provideboolean{shownotes} 
\setboolean{shownotes}{false} 
\newcommand{\margnote}[1]{
\ifthenelse{\boolean{shownotes}}%
{\marginpar{\raggedright\tiny\texttt{#1}}}%
{}%
}

\newcommand{\hole}[1]{
\ifthenelse{\boolean{shownotes}}%
{\begin{center} \fbox{ \rule {.25cm}{0cm}
\rule[-.1cm]{0cm}{.4cm} \parbox{.85\textwidth}{\begin{center}
\texttt{#1}\end{center}} \rule {.25cm}{0cm}}\end{center}}
{}
}

\theoremstyle{plain}

\newtheorem{lemma}{Lemma}[section]
\newtheorem{theorem}[lemma]{Theorem}

\newtheorem{corollary}[lemma]{Corollary}

\theoremstyle{definition}

\newtheorem{remark}[lemma]{Remark}
\newtheorem{definition}[lemma]{Definition}

\theoremstyle{remark}

\newcommand{\R}{\mathbb{R}}
\newcommand{\C}{\mathbb{C}}

\renewcommand{\Re}{\mathrm{Re}\,}

\newcommand{\<}{\langle}
\renewcommand{\>}{\rangle}

\newcommand{\bU}{\overline{U}}

\newcommand{\bpr}{\overline{p}_\rho}
\newcommand{\bpt}{\overline{p}_\theta}
\newcommand{\bet}{\overline{e}_\theta}

\newcommand{\bro}{\overline{\rho}}
\newcommand{\brt}{\overline{\theta}}
\newcommand{\bru}{\overline{u}}

\newcommand{\hU}{\widehat{U}}

\begin{document}

\title[Strict dissipativity of Cattaneo-Christov systems]{Strict dissipativity of Cattaneo-Christov systems for compressible fluid flow}

\author[F. Angeles]{Felipe Angeles}
 
\address{{\rm (F. Angeles)} Posgrado en Ciencias Matem\'{a}ticas\\Universidad Nacional Aut\'{o}noma 
de M\'{e}xico\\Circuito Exterior s/n, Ciudad de M\'{e}xico C.P. 04510 (Mexico)}

\email{teojkd@ciencias.unam.mx}

\author[C. M\'{a}laga]{Carlos M\'{a}laga}

\address{{\rm (C. M\'{a}laga)} Departamento de F\'{\i}sica\\Facultad de Ciencias\\Universidad Nacional Aut\'{o}noma de M\'{e}xico\\Circuito Exterior s/n, Ciudad Universitaria, Cd. Mx. C.P. 04510 (Mexico)}

\email{cmi@ciencias.unam.mx}

\author[R.G. Plaza]{Ram\'on G. Plaza}

\address{{\rm (R. G. Plaza)} Instituto de 
Investigaciones en Matem\'aticas Aplicadas y en Sistemas\\Universidad Nacional Aut\'onoma de 
M\'exico\\Circuito Escolar s/n, Ciudad de M\'{e}xico C.P. 04510 (Mexico)}

\email{plaza@mym.iimas.unam.mx}

\begin{abstract}
This work considers a compressible, viscous, heat-conducting fluid exhibiting thermal relaxation according to Christov's constitutive heat transfer law \cite{Chr09}, which is of Cattaneo type. The resulting evolution equations are known as Cattaneo-Christov systems. In this contribution, it is shown that Cattaneo-Christov systems for one-dimensional compressible fluid flow are strictly dissipative. The proof is based on the verification of a genuine coupling condition for hyperbolic-parabolic systems with viscous and relaxation effects combined as well as on showing the existence of compensating functions of the state variables in the sense of Shizuta and Kawashima \cite{ShKa85}. This property is used to obtain linear decay rates for solutions to the linearized equations around equilibrium states.
\end{abstract}



\maketitle

\setcounter{tocdepth}{1}

\section{Introduction}

One of the most important constitutive relations in continuum mechanics is Fourier's law of heat conduction. It states that, in an homogeneous, isotropic and thermally conducting medium, the heat flux $\boldsymbol{q}$ is determined by
\[
\boldsymbol{q} = - \kappa \nabla \theta,
\]
where $\theta = \theta(x,t)$ denotes the absolute temperature at a point $x$ of the medium at time $t > 0$ and $\kappa > 0$ is the thermal conductivity. Fourier's law is also a key ingredient in the compressible Navier-Stokes system of equations that describes the dynamics of a viscous compressible heat-conducting fluid (cf. \cite{Da4e}), inasmuch as the equation of conservation of energy underlies the heat transfer law proposed by Fourier. One of the main drawbacks of Fourier's constitutive law, however, is that it predicts infinite speed of propagation of heat, that is, thermal disturbances in a continuous medium will be felt instantly (although unequally) at all other points of the medium no matter how distant they are located. This unphysical behavior violates the well-established principle of \emph{causality} in continuum mechanics. 

Even though Fourier's law has been widely and successfully used to approximate the phenomenon of heat propagation in continuous media, other models have been proposed to correct this unrealistic feature. One of the best known is the \emph{Cattaneo-Maxwell heat transfer law} (see, e.g., \cite{JoPr89}),
\begin{equation}
\label{CMlaw}
\tau \boldsymbol{q}_t + \boldsymbol{q} = - \kappa \nabla \theta,
\end{equation}
where $\boldsymbol{q}_t = \partial \boldsymbol{q}/\partial t$ denotes the partial time-derivative of the heat flux and $\tau > 0$ is a constant. In the constitutive equation \eqref{CMlaw} (which can be traced back to the work of Maxwell \cite{Maxw1867} and was later reformulated by Cattaneo \cite{Catt49}), the parameter $\tau$ plays the role of an intrinsic relaxation time, or the time lag required for heat conduction to happen within a volume element once the temperature gradient has been established. Thus, this new term represents some sort of ``thermal inertia". Under this modification of Fourier's law, the flow of heat within the medium does not occur instantaneously but through the propagation of thermal waves with finite speed, a phenomenon known as \emph{second sound} (cf. \cite{Jrd14,LiSt78,Strau11}).

Even though Maxwell-Cattaneo heat transfer law preserves the causality principle for heat propagation in \emph{steady} continuous media, it is incompatible with the Galilean postulate of frame-indifference when the medium is in motion. Christov and Jordan \cite{ChJo05} have shown, for instance, that equation \eqref{CMlaw} violates this fundamental principle of classical mechanics and leads to paradoxical descriptions of the evolution of thermal waves. The reason is simple: thermal inertia should be a property of the material point, and the rate of change of the heat flux with respect of time must be the result of a change in the geometrical point (the partial time derivative) plus a change due to the transport of material quantities if the medium is in motion. Consequently, Christov and Jordan propose that the partial time derivative in \eqref{CMlaw} should be replaced by a \emph{material} derivative. Under this viewpoint, Christov \cite{Chr09} formulated a material, frame-indifferent version of the Cattaneo-Maxwell law that replaces the partial time derivative of the heat flux by a Lie-Oldroyd upper convected material derivative (cf. \cite{Old50}). It reads
\begin{equation}
\label{Christovlaw}
\tau \Big( \boldsymbol{q}_t + (\boldsymbol{u} \cdot \nabla) \boldsymbol{q} + (\nabla \cdot \boldsymbol{u}) \boldsymbol{q}- (\boldsymbol{q} \cdot \nabla) \boldsymbol{u} \Big) + \boldsymbol{q} = - \kappa \nabla \theta,
\end{equation}
where the vector $\boldsymbol{u} = \boldsymbol{u}(x,t)$ is velocity field of the medium. 

Coupling thermal relaxation of Cattaneo's type with the description of fluid flow has attracted the attention of the scientific community for a long time (see, e.g., \cite{CaMo72,JoPr89,LiSt78,Strau11}). In the description of moving continuous media, it is fundamental to consider a heat transfer law which, like \eqref{Christovlaw}, preserves the objectivity principle of Galileo. Hence, it is natural to couple Christov's constitutive law with the basic balance laws of mass, momentum and energy. The result is a system of equations of non-conservative type, due to the fact that now the heat flux variable $\boldsymbol{q}$ satisfies an evolutionary equation (the constitutive law) which does not express a balance law. In the context of moving continuous media, Christov's constitutive law \eqref{Christovlaw} has been recently investigated in the study of incompressible fluid flow \cite{TiZa11,Strau10b}, viscoelastic solids \cite{Morr10} and of compressible fluids with applications to acoustic wave propagation \cite{Jrd14,Strau10a}. As Straughan \cite{Strau10a} and Christov \cite{Chr09} point out, it is important to test this new model.

In this paper, we consider a compressible, viscous, heat-conducting fluid exhibiting thermal relaxation according to Christov's constitutive heat transfer law \eqref{Christovlaw}. We refer to the resulting equations as \textit{Cattaneo-Christov systems}. Our investigation focuses on one of the most important properties of this type of evolutionary systems of partial differential equations: their \emph{strict dissipativity}, or, in lay terms, the property that solutions to the linearized problem around equilibrium states show some decay structure (see the precise statement in section \ref{secgencoup} below). In physical terms, this property is tantamount to requiring that the dissipation terms do not allow solutions of traveling wave type to be, simultaneously, solutions to the associated hyperbolic system without dissipation. This characterization of strict dissipativity, known as \textit{genuine coupling}, has been extensively studied by Kawashima, Shizuta and collaborators \cite{Ka86,KaSh88b,KY04,ShKa85}. In the well-known case of the compressible Navier-Stokes system, for example, the terms due to viscosity and to thermal diffusion (Fourier's law) make the system strictly dissipative (cf. Shizuta and Kawashima \cite{ShKa85}). In the case of Cattaneo-Christov systems, thermal relaxation and viscous terms account for dissipation effects. Are these terms truly dissipative? Our contribution is to answer this question in the positive for one-dimensional Cattaneo-Christov systems. We establish strict dissipativity in two cases: on one hand, for the system of equations with viscosity and thermal relaxation terms combined, and on the other, for the inviscid counterpart in which thermal relaxation is the only dissipative term and the viscosity coefficients are set to zero. For that purpose, we first recast the system of equations as a quasi-linear system for which a symmetrizer can be found. Once the system is put into symmetric form, it is shown that the dissipation and the hyperbolic terms are genuinely coupled. Furthermore, we explicitly show the existence of \textit{compensating matrix functions} (cf. \cite{Hu05,ShKa85}) of the state variables which allow, in turn, to verify directly the strict dissipativity of the one-dimensional Cattaneo-Christov system and to establish energy estimates that yield the decay of solutions to the linearized problem around equilibrium states. These results apply to both the thermally relaxed fluid without viscosity and to the case with relaxation and viscosity effects combined.

\subsection*{Plan of the paper} In section \ref{secmodel} the Cattaneo-Christov model for one-dimensional compressible flow is introduced and recast as a system of equations in quasi-linear form. Section \ref{sechyp} contains the verification of the hyperbolicity of the system in the absence of dissipation terms and it is shown that Cattaneo-Christov systems are symmetrizable in one spatial dimension. In section \ref{secgencoup}, the genuine coupling condition of Kawashima and Shizuta, as well as the equivalence theorem for symmetric systems, are recalled. Moreover, it is shown that Cattaneo-Christov systems are genuinely coupled. Explicit forms of compensating functions in both the viscous and thermally relaxed cases are also provided via direct inspection. Section \ref{seclindec} contains the derivation of decay rates for solutions to the linearized system around equilibrium states. The paper ends with further discussion on the results and their possible extensions.

\section{Cattaneo-Christov systems for compressible fluid flow}
\label{secmodel}

Consider the basic equations for a compressible, viscous, heat conducting fluid in one dimensional space,
\begin{equation}
\label{eqbasics}
\begin{aligned}
\rho_t + (\rho u)_x &= 0,\\
(\rho u)_t + (\rho u^2)_x &= {\sigma}_x, \\
\Big( \rho ( e + \tfrac{1}{2} u^2)\Big)_t + \Big( \rho u ( e + \tfrac{1}{2} u^2)\Big)_x &= ({\sigma} u)_x - q_x,
\end{aligned}
\end{equation}
where $x \in \R$ and $t > 0$. According to custom, $\rho$ and $u$ denote the mass density and the velocity of the fluid, respectively, whereas $\sigma$ is the stress and has the form,
\begin{equation}
\label{eqstress}
\sigma = (2 \mu + \lambda) u_x - p.
\end{equation}
$p$ is the thermodynamic pressure, $e$ denotes the internal energy density, $q$ is the heat flux and the viscosity coefficients, $\lambda$ and $\mu$, satisfy the inequalities
\begin{equation}
\label{eqvisccoeff}
\mu \geq 0, \quad \tfrac{2}{3} \mu + \lambda \geq 0.
\end{equation}

The heat flux satisfies a constitutive relation that has the form of a heat transfer law. Instead of the usual Fourier's law, namely $q = - \kappa \theta_x$, where $\kappa = \kappa(\rho,\theta) > 0$ is the heat conductivity coefficient, we shall assume that the fluid, along with the property of conducting heat, exhibits thermal relaxation according to the following constituive equation,
\begin{equation}
\label{eqCC}
\tau \big(q_t +  u q_x\big) + q = - \kappa \theta_x,
\end{equation} 
where $\tau > 0$ is a constant characteristic relaxation time. Equation \eqref{eqCC} is a modification of the Cattaneo-Maxwell transfer law, namely, $\tau q_t + q = - \kappa \theta_x$. It is the one-dimensional version of the frame-indifferent material constitutive law \eqref{Christovlaw} proposed by Christov in \cite{Chr09}. The evolution of the flow is thus governed by the three balance laws for mass, momentum and energy \eqref{eqbasics} and the constitutive evolution equation \eqref{eqCC}. Of course, the system should be closed by an equation of the state for the fluid under consideration that determines the form of $p$ and $e$. In this paper, we make the following assumptions about the fluid:
\begin{equation}
\label{H1}
\tag{H$_1$}
\begin{minipage}[c]{4.0in}
The independent thermodynamic quantities are the mass density, $\rho > 0$, and the absolute temperature, $\theta > 0$. They vary within the domain
\[
\mathcal{D} := \{ (\rho, \theta) \in \R^2 \, : \, \rho > 0, \, \theta > 0 \}.
\]
\end{minipage}
\end{equation}

\bigskip

\begin{equation}
\label{H2}
\tag{H$_2$}
\begin{minipage}[c]{4.0in}
The pressure $p$, the internal energy density $e$, the heat conductivity coefficient $\kappa$ and the viscosity coefficients $\lambda$ and $\mu$ are smooth functions of $(\rho,\theta) \in \mathcal{D}$,
\[
p, e, \lambda, \mu, \kappa \in C^\infty(\mathcal{D}).
\]
In addition, $\lambda$ and $\mu$ satisfy inequalities \eqref{eqvisccoeff} and $\kappa > 0$ for all $(\rho,\theta) \in \mathcal{D}$.
\end{minipage}
\end{equation}

\bigskip

\begin{equation}
\label{H3}
\tag{H$_3$}
\begin{minipage}[c]{4.0in}
The fluid satisfies the following conditions:
\[
p> 0, \; \; p_\rho > 0, \;\; p_\theta > 0,\;\;e_\theta > 0, \qquad \text{for all} \;\; (\rho,\theta) \in \mathcal{D}.
\]
\end{minipage}
\end{equation}

\bigskip

Finally, for convenience in notation we define the combined viscosity coefficient $\nu \in C^\infty(\mathcal{D})$ as 
\[
\nu(\rho,\theta) := 2 \mu + \lambda.
\]
Notice that $\nu \geq 0$ on $\mathcal{D}$ in view of \eqref{eqvisccoeff}. 

\begin{remark}
Assumption \eqref{H3} is clearly satisfied by an ideal gas that satisfies Boyle's law,
\[
p(\rho, \theta) = R \rho \theta, \qquad e(\rho, \theta) = \frac{R  \theta}{\gamma -1},
\]
where $R > 0$ is the universal gas constant and $\gamma > 1$ is the adiabatic exponent. Hypotheses \eqref{H3} are, of course, more general and applicable to compressible fluids satisfying the standard assumptions of Weyl \cite{We49}, namely, adiabatic increase of pressure effects compression ($p_\rho > 0$), a generalized Gay-Lussac's law ($p_\theta > 0$) and the increase of internal energy due to an increase of temperature at constant volume ($e_\theta > 0$). 
\end{remark}

In this work we consider the basic equations \eqref{eqbasics} of conservation of mass, momentum and energy, coupled together with the evolution equation for the heat flux \eqref{eqCC}. As a result, we obtain the following quasi-linear system of equations 
\begin{equation}
\label{CCvisc}
\begin{aligned}
\rho_t + (\rho u)_x &= 0,\\
(\rho u)_t + (\rho u^2 + p)_x &=  \big( \nu u_x \big)_x, \\
\big( \rho ( e + \tfrac{1}{2} u^2)\big)_t + \big( \rho u ( e + \tfrac{1}{2} u^2)\big)_x &= (-p u + \nu u u_x)_x - q_x,\\
\tau q_t + \tau u q_x + q &= - \kappa \theta_x.
\end{aligned}
\end{equation}

In the case when $\nu > 0$ for all $(\rho,\theta) \in \mathcal{D}$, we call this system the \emph{viscous Cattaneo-Christov system for compressible fluid flow}. We shall distinguish between the viscous ($\nu > 0$) and the pure thermally relaxed system where $\nu \equiv 0$ for all $(\rho, \theta) \in \mathcal{D}$, that reads,
\begin{equation}
\label{CCrel}
\begin{aligned}
\rho_t + (\rho u)_x &= 0,\\
(\rho u)_t + (\rho u^2 + p)_x &=  0, \\
\big( \rho ( e + \tfrac{1}{2} u^2)\big)_t + \big( \rho u ( e + \tfrac{1}{2} u^2)\big)_x &= -(p u)_x - q_x,\\
\tau q_t + \tau u q_x + q &= - \kappa \theta_x.
\end{aligned}
\end{equation}
We denote system \eqref{CCrel} as the \emph{inviscid Cattaneo-Christov system}. These inviscid, thermally relaxed compressible fluids have been coined by Straughan as \emph{Cattaneo-Christov gases} \cite{Strau10a}.

In this paper we establish that, under the generic assumptions \eqref{H1} - \eqref{H3}, both systems are dissipative in a precise sense as we shall see below.

In the sequel, we denote $U = (\rho, u, \theta, q)^\top \in \mathcal{U} \subset \R^4$ as the vector of state variables, defined on the convex, open set
\begin{equation}
\label{defsetU}
\mathcal{U} := \{ (\rho, u, \theta, q)^\top \in \R^4 \, : \, \rho > 0, \, \theta > 0 \},
\end{equation}
known as the \emph{state space}.

Using the well-known thermodynamic relation $\theta p_\theta = p - \rho^2 e_\rho$ (see, e.g., \cite{Bt}, pg. 42) and after some algebra, we recast \eqref{CCvisc} as the following quasi-linear system for the state variables $U \in \mathcal{U}$,
\begin{equation}
\label{eqfullsyst}
A^0(U) U_t + A^1(U) U_x =  B(U) U_{xx}   + Q(U) + G(U, U_x),
\end{equation}
where
\[
A^0(U) := \begin{pmatrix}
1 & 0 & 0 & 0 \\ 0 & \rho & 0 & 0 \\ 0 & 0 & \rho e_\theta & 0 \\ 0 & 0 & 0 & \tau
\end{pmatrix}, \qquad 
A^1(U) := \begin{pmatrix}
u & \rho & 0 & 0 \\ p_\rho & \rho u & p_\theta & 0 \\ 0 &  \theta p_\theta & \rho u e_\theta & 1 \\ 0 & 0 & \kappa & \tau u
\end{pmatrix},
\]
\[
B(U) := \begin{pmatrix}
0 & 0 & 0 & 0 \\ 0 & \nu & 0 & 0 \\ 0 & 0 & 0 & 0 \\ 0 & 0 & 0 & 0
\end{pmatrix}, \qquad
Q(U) := \begin{pmatrix}
0 \\ 0 \\ 0 \\ -q
\end{pmatrix},
\] 
and $G(U,U_x)$ contains higher order (fully nonlinear) terms,
\begin{equation}
\label{defG}
G(U,U_x) = \begin{pmatrix}
0 \\ \nu_x u_x \\ \nu u_x^2 \\ 0 
\end{pmatrix} = O(|U_x|^2).
\end{equation}

Notice that $A^0, A^1, B \in C^\infty(\mathcal{U};\R^{4 \times 4})$, $Q \in C^\infty(\mathcal{U};\R^4)$, $G \in C^\infty(\mathcal{U} \times \R^4 ; \R^4)$. In view of hypotheses \eqref{H1} - \eqref{H3}, it is clear that for each $U \in \mathcal{U}$, $A^0(U) > 0$ is positive definite and hence, invertible, whereas $B(U) \geq 0$ is positive semi-definite. In the case where $\nu \equiv 0$ for all $(\rho, \theta) \in \mathcal{D}$ we recover the inviscid, thermally relaxed system \eqref{CCrel}, for which $B \equiv 0$.

\begin{remark}
Observe that the heat flux $q$ is regarded as a state variable and, thus, the constitutive heat transfer law \eqref{eqCC} is part of the time-dependent equations that determines the evolution of the system. As a result, system \eqref{eqfullsyst} is not expressed in conservation form. Instead, it is a quasi-linear, non-conservative system of equations with dissipation effects represented by viscosity (the term $B(U)U_{xx}$) and production terms due to relaxation (the thermal relaxation term $Q(U)$). 
\end{remark}

As in the study of systems of conservation laws with relaxation \cite{Da4e,L4}, the large time behavior of solutions is determined by a ``relaxed" structure, chosen so that the dynamics leads solutions towards an \emph{equilibrium manifold}. In quasilinear systems of the form \eqref{eqfullsyst}, the equilibrium manifold $\mathcal{V} \subset \R^4$ is defined as
\[
\mathcal{V} = \{ U \in \mathcal{U} \, : \, Q(U) = 0\}.
\]
Mimicking discrete kinetic theory \cite{Ka84}, the \emph{space of collision invariants} is defined as
\[
\mathcal{M} = \{ \psi \in \R^4 \, : \, \psi^\top Q(U) = 0, \; \text{for any } \, U \in \mathcal{U}\} \subset \R^4.
\]
A solution $U = U(x,t)$ to system \eqref{eqfullsyst} is an \emph{equilibrium solution} (or a \emph{Maxwellian}) if it lies on the equilibrium manifold, that is, if $Q(U(x,t)) = 0$ for all $x \in \R$, $t > 0$. Clearly, any constant state in the equilibrium manifold, $\overline{U} \in \mathcal{V}$, is an equilibrium solution. In the case of the Cattaneo-Christov system \eqref{eqfullsyst} the equilibrium manifold is given by
\begin{equation}
\label{defV}
\mathcal{V} = \{ (\rho, u, \theta, q)^\top \in \R^4 \, : \, \rho > 0, \, \theta > 0, \, q = 0 \},
\end{equation}
that is, it corresponds to the states with zero heat flux. Also particular to the Cattaneo-Christov system is the following property, $\mathcal{V} = \mathcal{M} \cap \mathcal{U}$, as the reader may easily verify.

\section{Hyperbolicity and symmetrizability}
\label{sechyp}

\subsection{Hyperbolicity}

Let us consider the system
\begin{equation}
\label{hypsyst}
A^0(U) U_t + A^1(U) U_x = 0,
\end{equation}
which results from neglecting thermal relaxation and dissipation due to viscosity in \eqref{eqfullsyst}. For any state $U \in \mathcal{U}$, \eqref{hypsyst} is a quasi-linear, strictly hyperbolic first order system. Although hyperbolicity has been mentioned before as a property of this ``inviscid" Cattaneo-Christov system in one dimension  (see, for instance, \cite{Jrd14} and the references therein), for the sake of completeness we verify this fact by computing its characteristic speeds which (apparently) have not been reported before in the literature. For any $U \in \mathcal{U}$, set
\begin{equation}
\label{chardet}
\pi(\zeta) = \det \Big( A^1(U) - \zeta A^0(U) \Big).
\end{equation}

The roots of $\pi(\zeta) = 0$ are called the \emph{characteristic speeds} of system \eqref{hypsyst}. If these roots are all real and different then it is said that the system \eqref{hypsyst} is strictly hyperbolic at $U \in \mathcal{U}$. 
\begin{remark}
We remind the reader that the notion of hyperbolicity is motivated by the existence of traveling wave solutions to system \eqref{hypsyst} of the form $U(x,t) = \varphi(x - st)$, for some real propagating speed $s \in \R$ and a profile vector function $\varphi$. Substitution yields the spectral problem
\begin{equation}
\label{tws}
(A^1(\varphi) - s A^0(\varphi) ) \varphi' = 0,
\end{equation}
with eigenvalue $s \in \R$ and eigenfunction $\varphi'$, which leads directly to the characteristic equation \eqref{chardet}. 
\end{remark}
After a straightforward computation we see that
\[
\pi(\zeta) = \det \begin{pmatrix}
u - \zeta & \rho & 0 & 0 \\ p_\rho & \rho(u - \zeta) & p_\theta & 0 \\ 0 & \theta p_\theta & \rho e_\theta (u - \zeta) & 1 \\ 0 & 0 & \kappa & \tau (u - \zeta)
\end{pmatrix}.
\]
Let us denote $m = u - \zeta$ and make the computations to arrive at
\[
\pi(\zeta) = \rho (m^2 - p_\rho) (\rho e_\theta \tau m^2 - \kappa) - \theta p_\theta^2 \tau m^2.
\]
This is a second order polynomial in $m^2$. Therefore, we have that $\pi(\zeta) = 0$ if and only if
\[
m^4 + \widetilde{b} m^2 + \widetilde{c} = 0,
\]
where
\[
\widetilde{b} = - (\rho^2 e_\theta \tau)^{-1} ( \rho \kappa + \rho^2 p_\rho e_\theta \tau + \theta p_\theta^2 \tau), \qquad \widetilde{c} = (\rho^2 e_\theta \tau)^{-1} \rho p_\rho \kappa.
\]
Upon inspection of the discriminant
\[
\begin{aligned}
\Delta = \widetilde{b}^2 - 4 \widetilde{c} &= \left( p_\rho + \frac{\kappa}{\rho e_\theta \tau} + \frac{\theta p_\theta^2}{\rho^2 e_\theta} \right)^2 - \frac{4 \kappa p_\rho}{\rho e_\theta \tau} \\
&= \left( p_\rho - \frac{\kappa}{\rho e_\theta \tau} \right)^2 + \frac{\theta p_\theta^2}{\rho^2 e_\theta} \left( 2 p_\rho + \frac{2 \kappa}{\rho e_\theta \tau} + \frac{\theta p_\theta^2}{\rho^2 e_\theta}\right) > 0,
\end{aligned}
\]
we conclude that the $m^2$-roots are real and positive,
\[
0 < m_-^2 = \tfrac{1}{2} |\widetilde{b}| - \tfrac{1}{2} \sqrt{\, \widetilde{b}^2 - 4 \widetilde{c} \; } < m_+^2 = \tfrac{1}{2} |\widetilde{b}| + \tfrac{1}{2} \sqrt{\, \widetilde{b}^2 - 4 \widetilde{c} \; },
\]
yielding the characteristic speeds
\[
\zeta_1 = u - \sqrt{m_+^2} < \zeta_2 = u - \sqrt{m_-^2} < \zeta_3 = u + \sqrt{m_-^2} < \zeta_4 = u + \sqrt{m_+^2}.
\]
We conclude that system \eqref{hypsyst} is strictly hyperbolic. We gather these observations into the following
\begin{lemma}
\label{lemcharspeed}
Under assumptions \eqref{H1} - \eqref{H3} and for each $U = (\rho, u, \theta, q)^\top \in \mathcal{U} \subset \R^4$, the first order system \eqref{hypsyst} is strictly hyperbolic at $U \in \mathcal{U}$ and the characteristic speeds are given by
\[
\begin{aligned}
\zeta_1(U) &= u - \frac{1}{\sqrt{2}} {\sqrt{ p_\rho + \frac{\kappa}{\rho e_\theta \tau} + \frac{\theta p_\theta^2}{\rho^2 e_\theta} \, + \, \sqrt{\left( p_\rho + \frac{\kappa}{\rho e_\theta \tau} + \frac{\theta p_\theta^2}{\rho^2 e_\theta} \right)^2 - \frac{4 \kappa p_\rho}{\rho e_\theta \tau} \;\;} \;\;}}, \\
\zeta_2(U) &= u - \frac{1}{\sqrt{2}} {\sqrt{ p_\rho + \frac{\kappa}{\rho e_\theta \tau} + \frac{\theta p_\theta^2}{\rho^2 e_\theta} \, - \, \sqrt{\left( p_\rho + \frac{\kappa}{\rho e_\theta \tau} + \frac{\theta p_\theta^2}{\rho^2 e_\theta} \right)^2 - \frac{4 \kappa p_\rho}{\rho e_\theta \tau} \;\;} \;\;}}, \\
\zeta_3(U) &= u + \frac{1}{\sqrt{2}} {\sqrt{ p_\rho + \frac{\kappa}{\rho e_\theta \tau} + \frac{\theta p_\theta^2}{\rho^2 e_\theta} \, - \, \sqrt{\left( p_\rho + \frac{\kappa}{\rho e_\theta \tau} + \frac{\theta p_\theta^2}{\rho^2 e_\theta} \right)^2 - \frac{4 \kappa p_\rho}{\rho e_\theta \tau} \;\;} \;\;}}, \\
\zeta_4(U) &= u + \frac{1}{\sqrt{2}} {\sqrt{ p_\rho + \frac{\kappa}{\rho e_\theta \tau} + \frac{\theta p_\theta^2}{\rho^2 e_\theta} \, + \, \sqrt{\left( p_\rho + \frac{\kappa}{\rho e_\theta \tau} + \frac{\theta p_\theta^2}{\rho^2 e_\theta} \right)^2 - \frac{4 \kappa p_\rho}{\rho e_\theta \tau} \;\;} \;\;}}.
\end{aligned}
\]
\end{lemma}

In the case of the standard model for inviscid compressible fluid flow (namely, Euler equations), it is well-known \cite{Da4e,Smo94} that the three characteristic speeds (in one spatial dimension) are $u-c$, $u$ and $u + c$, where the positive quantity $c = \sqrt{p_\rho} > 0$ is known as \emph{the speed of sound}. In the present case we have two ``sound speeds", $c_1 = \sqrt{m_+^2}$ and $c_2 = \sqrt{m_-^2}$, and the characteristic speeds of the system split into $u-c_2 < u - c_1 < u + c_1 < u+ c_2$. These sound speeds convey both thermal and mechanical contributions due to the rate of change of the pressure with respect to changes in density and in temperature, respectively. Notice that when thermal effects are neglected, formally, in the limit when $\kappa \to 0^+$ and $p_\theta \to 0^+$, we have that $c_1, c_2 \to \sqrt{p_\rho}$, and the two sound speeds converge to the sole mechanical sound speed $c$ (the absence of thermal waves). On the other hand, if we take the (non-rigorous) limit when $p_\rho \to 0^+$ and $p_\theta \to 0^+$ then $c_1 \to 0$ and $c_2 \to \sqrt{\kappa/(\rho e_\theta \tau)}$; this last value is the thermal wave speed in the absence of mechanical effects as computed by Lindsay and Straughan (see equation (4.29) in \cite{LiSt78}; see also \cite{Strau10a}). The significance of the characteristic speeds of Lemma \ref{lemcharspeed} is that they comprise the exact way in which mechanical and thermal effects are combined.

\subsection{Symmetrizability}

We now show that system \eqref{eqfullsyst} can be put into symmetric form. Let us denote
\[
D(U) := d_U Q(U) = \begin{pmatrix} 0 & 0 & 0 & 0 \\ 0 & 0 & 0 & 0 \\ 0 & 0 & 0 & 0 \\ 0 & 0 & 0 & -1 \end{pmatrix}, \qquad U \in \mathcal{U}.
\]
\begin{definition}
We say a quasilinear system of the form \eqref{eqfullsyst} is \emph{symmetrizable} provided that there exists a matrix function $S \in C^\infty(\mathcal{U}; \R^{4 \times 4})$, $S = S(U)$, symmetric and positive definite, such that the matrices $S(U)A^0(U), S(U)A^1(U), S(U)B(U)$ and $S(U)D(U)$ are symmetric for all $U \in \mathcal{U}$.
\end{definition}

\begin{lemma}
Under assumptions \eqref{H1} - \eqref{H3}, Cattaneo-Christov system \eqref{eqfullsyst} is symmetrizable and the symmetrizer $S \in C^\infty(\mathcal{U}; \R^{4 \times 4})$ is given by
\begin{equation}
\label{symm}
S(U) := \begin{pmatrix}
 \displaystyle{\frac{p_\rho}{\rho}} & 0 & 0 & 0 \\ 0 & 1 & 0 & 0 \\ 0 & 0 & \displaystyle{\frac{1}{\theta}} & 0 \\ 0 & 0 & 0 &  \displaystyle{\frac{1}{\kappa \theta}}
\end{pmatrix}, \qquad U \in \mathcal{U}.
\end{equation}
\end{lemma}
\begin{proof}
Clearly, $S$ is smooth in the convex open set $\mathcal{U}$. Moreover, $S$ is symmetric (diagonal) and positive definite in view of \eqref{H1} - \eqref{H3}. That $S$ symmetrizes system \eqref{eqfullsyst} follows from straighforward computations that yield
\begin{equation}
\label{defhA0}
\hat{A}^0(U) := S(U) A^0(U) = \begin{pmatrix}
\displaystyle{\frac{p_\rho}{\rho}} & 0 & 0 & 0 \\ 0 & \rho & 0 & 0 \\ 0 & 0 &  \displaystyle{\frac{\rho e_\theta}{\theta}}  & 0 \\ 0 & 0 & 0 & \displaystyle{\frac{\tau}{\kappa \theta}} \end{pmatrix},
\end{equation}
\begin{equation}
\label{defhA1}
\renewcommand\arraystretch{2}
\hat{A}^1(U) := S(U) A^1(U) = \begin{pmatrix}
\displaystyle{\frac{u p_\rho}{\rho}} & p_\rho & 0 & 0 \\ p_\rho & \rho u & p_\theta & 0 \\ 0 & p_\theta & \displaystyle{\frac{\rho u e_\theta}{\theta}}  & \displaystyle{\frac{1}{\theta}} \\ 0 & 0 & \displaystyle{\frac{1}{\theta}} &  \displaystyle{\frac{\tau u}{\kappa \theta}} 
\end{pmatrix},
\end{equation}
\begin{equation}
\label{defhB}
\hat{B}(U) := S(U) B(U) = \begin{pmatrix}
0 & 0 & 0 & 0 \\ 0 & \nu & 0 & 0 \\ 0 & 0 & 0 & 0 \\ 0 & 0 & 0 & 0
\end{pmatrix}, 
\end{equation}
\begin{equation}
\label{defhD}
\hat{D}(U) := S(U) D(U) = \begin{pmatrix} 
0 & 0 & 0& 0\\ 0& 0 & 0& 0\\ 0& 0& 0 & 0\\ 0& 0& 0& - \displaystyle{\frac{1}{\kappa \theta}}
\end{pmatrix},
\end{equation}
which are smooth symmetric matrix functions of $U \in \mathcal{U}$.
\end{proof}

\begin{remark}
\label{remnoentropy}
It is well-known \cite{Da4e} that symmetrizability implies hyperbolicity of system \eqref{hypsyst}. Also, since the works of Friedrichs \cite{Frd54} and Goudunov \cite{Godu61a}, symmetrizability has established itself as an important property. It plays a key role, for example, to perform energy estimates and to study existence and stability of solutions. For systems in conservation form the symmetrizer is usually the Hessian of a convex entropy function. Even in the case of quasi-linear systems not in conservation form (where the coefficients $A^j$ are not necessarily Jacobians of flux functions $f^j$) it is possible to define a convex entropy, as shown by Kawashima and Yong \cite{KY04}: if the symmetrizer is the Jacobian of a diffeomorphic change of variables, $S(U) = D_U \Psi(U)$, then a convex entropy function can be introduced. For Cattaneo-Christov systems, however, the symmetrizer \eqref{symm} is not the Jacobian of a particular diffeomorphism and the system is not necessarily endowed with a convex entropy function.
\end{remark}

\section{The genuine coupling condition}
\label{secgencoup}

\subsection{Strict dissipativity and genuine coupling}

In order to define the strict dissipativity of the system, let us consider solutions around a constant equilibrium state
\[
\bU = (\overline{\rho}, \overline{u}, \overline{\theta}, 0)^\top \in \mathcal{V},
\]
for which $Q(\overline{U}) = 0$. If $\bU + U$ is a solution to \eqref{eqfullsyst} then we can recast the system as
\[
A^0(\bU) U_t + A^1(\bU) U_x = B(\bU) U_{xx} + D(\bU) U + \mathcal{N}(U, U_x, U_t),
\]
where $\mathcal{N}$ comprises the nonlinear terms. Multiply on the left by the constant, symmetric, positive definite matrix $S(\overline{U})$ to arrive at the following symmetric system
\begin{equation}
\label{nlsyst}
A^0 U_t + A^1 U_x + L U = B U_{xx} + \overline{\mathcal{N}},
\end{equation}
where,
\[
\begin{aligned}
A^0 &:= S(\overline{U}) A^0(\bU) = \hat{A}^0 (\bU),\\
A^1 &:= S(\overline{U}) A^1(\bU) = \hat{A}^1 (\bU),\\
B &:= S(\bU) B(\bU) = \hat{B}(\bU),\\
L &:= - S(\bU) D(\bU) = \hat{D}(\bU), 
\end{aligned}
\]
and, once again, $\overline{\mathcal{N}} = S(\overline{U}) \mathcal{N}$ contains the nonlinear terms. Notice that $A^j$, $j =0,1$, $B$ and $L$ are real symmetric constant matrices, with $A^0 > 0$ (positive definite) and $B$, $L \geq 0$ (positive semi-definite). 

Let us consider the linear part of \eqref{nlsyst}, namely, the linear symmetric system
\begin{equation}
\label{lind}
A^0 U_t + A^1 U_x + L U = B U_{xx}, 
\end{equation}
which is the symmetric version of \eqref{eqfullsyst}, linearized around an equilibrium state $\bU \in \mathcal{V}$. Since it is a system with constant coefficients the solution can be determined by its Fourier transform with respect to the spatial variable $x \in \R$. The resulting equation is
\begin{equation}
\label{Foulin}
A^0 \hU_t + i \xi A^1 \hU + L \hU + \xi^2 B \hU = 0, \qquad t > 0, \;\, \xi \in \R,
\end{equation}
where $\hU = \hU(\xi,t)$ denotes the Fourier transform of $U$.

The fact that $A^0 > 0$ and $L,B \geq 0$ is not enough to guarantee the decay of solutions to the linear problem \eqref{lind}. We resort to the following sufficient condition for the essential spectrum of the linear constant coefficient differential operator to be stable.  For each $\xi \in \R$, $\xi \neq 0$, let $\lambda = \lambda(\xi) \in \C$ denote the eigenvalues of the corresponding characteristic equation, namely, the roots of the following dispersion relation,
\begin{equation}
\label{chareq}
\det \big( \lambda A^0  + i \xi A^1 + L + \xi^2 B \big) = 0.
\end{equation}

\begin{definition}[strict dissipativity]
System \eqref{lind} is said to be \emph{strictly dissipative} if $\Re \lambda(\xi) < 0$ for all $\xi \in \R$, $\xi \neq 0$.
\end{definition}

Closely related to the dissipativity condition is the following
\begin{definition}[genuine coupling]
System \eqref{lind} satisfies the \emph{genuine coupling condition} at any state $\bU \in \mathcal{U}$ if for any $V \in \R^4$, $V \neq 0$, with $BV = LV = 0$ then we have that $(\lambda A^0 + A^1)V \neq 0$ for all $\lambda \in \R$.
\end{definition}
\begin{remark}
This condition basically expresses that no eigenvector of the hyperbolic part of the operator lies in the kernel of the dissipative terms. Such property is physically relevant. For instance, loss of genuine coupling results into hyperbolic directions whereby traveling wave solutions to system \eqref{hypsyst} are not dissipated by the viscous and relaxation terms. In other words, wave solutions to \eqref{hypsyst} (hence satisfying the spectral equation \eqref{tws}) are also solutions to \eqref{lind} if the eigenvector $\varphi'$ lies in $\ker B \cap \ker L$. Genuine coupling has also deep consequences on the time asymptotic smoothing behavior of solutions to viscous and relaxation systems of conservation laws (see, for example, \cite{Hoff92}). This condition is also known in the literature as \emph{the Kawashima-Shizuta condition}, or simply, \emph{the Kawashima condition} (see \cite{LoRu06,MaN2,RuSe04} and some of the references therein).
\end{remark}

Let us now recall the concept of a compensating function in the sense of Kawashima and Shizuta \cite{ShKa85}, specialized to the present one-dimensional case.
\begin{definition}
\label{defK}
A matrix $K$ is a \emph{compensating function} for system \eqref{lind} provided that
\begin{itemize}
\item[(a)] $K A^0$ is skew-symmetric, and
\item[(b)] $\tfrac{1}{2} \big( KA^1 + (KA^1)^\top) + B + L$ is positive definite.
\end{itemize}
\end{definition}

In the case of symmetric systems, the properties of genuine coupling, strict dissipativity and the existence of a compensating function are equivalent. This fact was first proved by Shizuta and Kawashima \cite{ShKa85} and fully characterizes the stability condition for system \eqref{lind} in the symmetric case (see also Humpherys \cite{Hu05} for an extension to higher order systems).
\begin{theorem}[Shizuta-Kawashima \cite{ShKa85}]
\label{theoequiv}
Assume $A^j, B, L$, $j =0,1$, are real symmetric matrices, with $A^0 > 0$, $B, L \geq 0$. Then the following statements are equivalent:
\begin{itemize}
\item[(a)] System \eqref{lind} is strictly dissipative.
\item[(b)] System \eqref{lind} satisfies the genuine coupling condition at $\bU \in \mathcal{U}$.
\item[(c)] There exists a compensating function $K$ for system \eqref{lind}.
\item[(d)] There exists a positive constant $k > 0$ such that for any $\xi \in \R$, $\xi \neq 0$, and any root $\lambda = \lambda(\xi)$ of the characteristic equation \eqref{chareq} there holds
\begin{equation}
\label{lambdabd}
\Re \lambda(\xi) \leq - \frac{k \xi^2}{1 + \xi^2}.
\end{equation}
\end{itemize}
\end{theorem}
\begin{remark}
Notice that property (d) implies automatically property (a). It is easy to prove that genuine coupling is a necessary condition for strict dissipativity, i.e. that (a) implies (b). The equivalence theorem establishes the existence of a compensating function once the genuine coupling condition has been verified. It is worth mentioning that the general proof in \cite{ShKa85} (see also \cite{Hu05}) is constructive. It provides a formula for $K$ in terms of the eigenprojections of the hyperbolic part ($K$ is, in fact, a Drazin inverse of the conmutator operator; see Humpherys \cite{Hu05} for further information).
\end{remark}

\subsection{Genuine coupling of Cattaneo-Christov systems}

We now show that Cattaneo-Christov systems are genuinely coupled. In the sequel, for any fixed state $\bU = (\bro, \bru, \brt, \overline{q})^\top \in \mathcal{U}$ we shall denote
\[
\overline{p} := p(\bro,\brt), \;\;\; \overline{e} := e(\bro,\brt),\;\;\; \overline{\kappa} := \kappa(\bro,\brt) \;\;\; \overline{\nu} := \nu (\bro,\brt),
\]
\[
\bpr := p_\rho(\bro,\brt), \;\;\; \bpt := p_\theta(\bro,\brt), \;\;\;  \bet := e_\theta(\bro,\brt).
\]

\begin{lemma}
Under assumptions \eqref{H1} - \eqref{H3}, Cattaneo-Christov systems \eqref{eqfullsyst} satisfy the genuine coupling condition at any fixed state $\bU = (\bro, \bru, \brt, \overline{q})^\top \in \mathcal{U}$.
\end{lemma}
\begin{proof}
As before, we denote $A^j = \hat{A}^j(\bU)$, $B = \hat{B}(\bU)$, $L = - \hat{D}(\bU)$, $j =0,1$. From the expression for $L$ in \eqref{defhD}, we see that any $V \in \ker L$ is of the form $V = (v_1, v_2, v_3, 0)^\top$, with $v_j \in \R$. Therefore, from \eqref{defhA0} and \eqref{defhA1} and for any $\lambda \in \R$ we have
\[
(\lambda {A}^0 + {A}^1) V = \begin{pmatrix} 
\displaystyle{\frac{\bpr}{\bro}(\lambda + \bru) v_1} + \bpr v_2 \\
\bpr v_1 + \bro(\lambda + \bru)v_2 +  \bpt v_3  \\  \displaystyle{\frac{\bro \, \bet}{\brt}(\lambda + \bru) v_3 + \bpt v_2}\\  \displaystyle{ \frac{v_3}{\brt}}
\end{pmatrix}.
\] 

Suppose that $V \in \ker L$, $V \neq 0$ and $(\lambda {A}^0 + {A}^1) V = 0$ for some $\lambda \in \R$. From $\brt > 0$ we deduce that $v_3 = 0$. This yields $v_2 = 0$ as $\bpt > 0$. Finally, from $\bpr > 0$ we get $v_1 = 0$. Thus, we conclude that $V = 0$, a contradiction.
\end{proof}

\begin{remark}
It is to be observed that the genuine coupling condition holds at any state $\bU \in \mathcal{U}$ (not necessarily an equilibrium state). Also, notice that both the viscous, thermally relaxed Cattaneo-Christov system \eqref{CCvisc} with $\nu > 0$ and the relaxation system \eqref{CCrel} with $\nu \equiv 0$, are genuinely coupled. Indeed, in the viscous case with $V \in \ker B \cap \ker L$ the proof is exactly the same.
\end{remark}

Although genuine coupling readily implies the existence of a compensating function (thanks to Theorem \ref{theoequiv}), it is often possible to provide a formula for it by direct inspection. 

\begin{lemma}
\label{lemviscK}
Under assumptions \eqref{H1} - \eqref{H3} and in the viscous case ($\nu > 0$ for all $(\rho, \theta) \in \mathcal{D}$), for every equilibrium state $\bU \in \mathcal{V}$ there exists a compensating function for system \eqref{lind}, which is given explicitly by
\begin{equation}
\label{viscK}
K = \delta \begin{pmatrix} 0 & \bpr & 0 & 0 \\ - \bpr & 0 & - \bpt & 0 \\ 0 & \bpt & 0 & 0 \\ 0 & 0 & 0 & 0 \end{pmatrix} \big( A^0 \big)^{-1},
\end{equation}
for some $0 < \delta \ll 1$ sufficiently small.
\end{lemma}
\begin{proof}
We verify directly that \eqref{viscK} is a compensating function for system \eqref{lind}. First observe from expression \eqref{viscK} that $KA^0$ is clearly skew-symmetric. Let us now compute
\[
\begin{aligned}
KA^1 &= 
\delta \begin{pmatrix} 0 & \bpr & 0 & 0 \\ - \bpr & 0 & - \bpt & 0 \\ 0 & \bpt & 0 & 0 \\ 0 & 0 & 0 & 0 \end{pmatrix} \begin{pmatrix} \displaystyle{\frac{\bro}{\bpr}} & 0& 0& 0\\ 0& \displaystyle{\frac{1}{\bro}} & 0& 0\\ 0& 0& \displaystyle{\frac{\brt}{\bro \bet}} &0 \\ 0& 0& 0& \displaystyle{\frac{\bar{\kappa} \brt}{\tau}}\end{pmatrix} \renewcommand\arraystretch{2} \begin{pmatrix} \displaystyle{\frac{\bru \, \bpr}{\bro}} & \bpr &  0& 0\\ \bpr  & \bro \, \bru & \bpt & 0\\ 0& \bpt & \displaystyle{\frac{\bro \, \bru \, \bet}{\brt}} & \displaystyle{\frac{1}{\brt}} \\ 0& 0&  \displaystyle{\frac{1}{\brt}} & \displaystyle{\frac{\tau \bru}{\bar{\kappa} \brt}}\end{pmatrix}
\\
&= 
\delta \renewcommand\arraystretch{2} \begin{pmatrix} \displaystyle{\frac{\bpr^2}{\bro}} & \bru \, \bpr & \displaystyle{\frac{\bpr \, \bpt}{\bro}} & 0 \\ -\bru \, \bpr & \displaystyle{- \Big( \bro \, \bpr + \frac{\brt \, \bpt^2}{\bro \, \bet}\Big)} & -\bru \bpt & \displaystyle{- \, \frac{\bpt}{\bro \, \bet}} \\ \displaystyle{\frac{\bpr \, \bpt}{\bro}} &  \bru \, \bpt & \displaystyle{\frac{\bpt^2}{\bro}} & 0 \\ 0 & 0 & 0 & 0 \end{pmatrix}.
\end{aligned}
\]

Its symmetric part is
\[
\tfrac{1}{2} \big( KA^1 + (KA^1)^\top \big) = \delta \renewcommand\arraystretch{2} \begin{pmatrix} \displaystyle{\frac{\bpr^2}{\bro}} & 0 & \displaystyle{\frac{\bpr \, \bpt}{\bro}} & 0 \\ 0 & \displaystyle{- \Big( \bro \, \bpr + \frac{\brt \, \bpt^2}{\bro \, \bet}\Big)} & 0 & \displaystyle{- \, \frac{\bpt}{2 \bro \, \bet}} \\ \displaystyle{\frac{\bpr \, \bpt}{\bro}} &  0 & \displaystyle{\frac{\bpt^2}{\bro}} & 0 \\ 0 & \displaystyle{- \, \frac{\bpt}{2 \bro \, \bet}} & 0 & 0 \end{pmatrix}.
\]

Therefore, for any $X = (x_1, \, x_2, \, x_3, \, x_4)^\top \in \R^4$, $X \neq 0$, we have the following quadratic form
\[
\begin{aligned}
Q(X) &:= X^\top \Big( \tfrac{1}{2} \big( KA^1 + (KA^1)^\top \big) + B + L \Big) X  \\ 
&=\delta \frac{\bpr^2}{\bro} x_1^2 + 2\delta \frac{\bpt \bpr}{\bro} x_1 x_3 - \delta \frac{\bpt}{\bro \, \bet} x_2 x_4 + \delta \frac{\bpt^2}{\bro} x_3^2  + \Big( \bar{\nu} - \delta \Big( \bro \, \bpr + \frac{\brt \, \bpt^2}{\bro \, \bet}\Big)\Big) x_2^2 + \frac{1}{\bar{\kappa} \brt} x_4^2
\\
&\geq \frac{\delta}{2} \frac{\bpr^2}{\bro} x_1^2 + \delta \frac{\bpt^2}{\bro} x_3^2 + \Big( \bar{\nu} - \delta \Big( \bro \, \bpr + \frac{\brt \, \bpt^2}{\bro \, \bet} + \frac{\bpt}{2 \bro \, \bet}\Big)\Big) x_2^2 + \Big( \frac{1}{\bar{\kappa} \brt} - \delta \frac{\bpt}{2 \bro \, \bet} \Big) x_4^2.
\end{aligned}
\]

Thanks to hypotheses \eqref{H1} - \eqref{H3} and since $\bar{\nu} > 0$, one can choose $\delta > 0$ sufficiently small such that
\[
0 < \delta < \frac{2 \bro \, \bet}{\bar{\kappa} \brt \bpt} \quad \text{and} \quad 0 < \delta < \bar{\nu} \Big( \bro \, \bpr + \frac{\brt \, \bpt^2}{\bro \, \bet} + \frac{\bpt}{2 \bro \, \bet}\Big)^{-1},
\]
yielding
\[
Q(X) \geq C_\delta |X|^2 > 0,
\]
for some $C_\delta > 0$ and all $X \neq 0$.
\end{proof}

In the case without viscosity the form of $K$ differs considerably, due to the fact that the only dissipation term is the thermal relaxation one.
\begin{lemma}
\label{remrelaxK}
Under assumptions \eqref{H1} - \eqref{H3} and in the pure thermal relaxation case ($\nu \equiv 0$ for all $(\rho, \theta) \in \mathcal{D}$), for every equilibrium state $\bU \in \mathcal{V}$ there exists a compensating function for system \eqref{lind}, which is given explicitly by
\begin{equation}
\label{relaxK}
K = \delta \renewcommand\arraystretch{2} \begin{pmatrix}   0 & \displaystyle{\frac{\delta^2 \tau \brt^2 \bpt^2 \bpr}{\bro^2}} & 0 & 0 \\ - \displaystyle{\frac{\delta^2 \tau \brt^2 \bpt^2 \bpr}{\bro^2}} & 0 & \delta \bpt & 0 \\ 0 & - \delta \bpt & 0 & \displaystyle{\frac{\bro \bet}{\bar{\kappa} \brt^2}} \\
0 & 0 & - \displaystyle{\frac{\bro \bet}{\bar{\kappa} \brt^2}} & 0\end{pmatrix} \big( A^0 \big)^{-1},
\end{equation}
for some $0 < \delta \ll 1$ sufficiently small.
\end{lemma}
\begin{proof}
We propose to take $K$ of the form
\[
K = \begin{pmatrix}
0 & \alpha & 0 & 0 \\ - \alpha & 0 & - \beta & 0 \\ 0 & \beta & 0 & - \gamma \\ 0 & 0 & \gamma & 0 
\end{pmatrix}\big( A^0 \big)^{-1},
\]
and to appropriately choose constants $\alpha, \beta$ and $\gamma$. Performing the product yields the matrix
\[
K A^1 = \renewcommand\arraystretch{2} \begin{pmatrix}
\displaystyle{\frac{\alpha \bpr}{\bro}} & \alpha \bar{u} & \displaystyle{\frac{\alpha \bpt}{\bro}} & 0 \\ - \alpha \bar{u} & - \displaystyle{\big( \alpha \bro + \frac{\beta \brt \bpt}{\bro \bet}\big)} & - \beta \bar{u} & - \displaystyle{\frac{\beta}{\bro \bet}} \\
\displaystyle{\frac{\beta \bpr}{\bro}} & \beta \bar{u} & \displaystyle{\frac{\beta \bpt}{\bro} - \frac{\gamma \bar{\kappa}}{\tau}} & - \gamma \bar{u} \\
0 & \displaystyle{\frac{\gamma \brt \bpt}{\bro \bet}} & \gamma \bar{u} & \displaystyle{\frac{\gamma}{\bro \bet}}
\end{pmatrix},
\]
whose symmetric part is
\[
\begin{aligned}
\tfrac{1}{2} \big( KA^1 &+ (KA^1)^\top \big) = \\ &=\begin{pmatrix}
\displaystyle{\frac{\alpha \bpr}{\bro}} & 0 & \displaystyle{\frac{1}{2\bro} \big( \beta \bpr + \alpha \bpt \big)} & 0 \\ 
0 & - \displaystyle{\big( \alpha \bro + \frac{\beta \brt \bpt}{\bro \bet}\big)} & 0 & \displaystyle{\frac{1}{2\bro \bet} \big( \gamma \brt \bpt - \beta\big)} \\
\displaystyle{\frac{1}{2\bro} \big( \beta \bpr + \alpha \bpt \big)} & 0 & \displaystyle{\frac{\beta \bpt}{\bro} - \frac{\gamma \bar{\kappa}}{\tau}} &0 \\
0 & \displaystyle{\frac{1}{2\bro \bet} \big( \gamma \brt \bpt - \beta\big)} & 0 & \displaystyle{\frac{\gamma}{\bro \bet}}
\end{pmatrix}.
\end{aligned}
\]

Thus, in view that $B = 0$, we have for any $X = (x_1, x_2, x_3, x_4)^\top$, $X \neq 0$, that the corresponding quadratic form is
\[
\begin{aligned}
Q(X) &:= X^\top \big( \tfrac{1}{2} ( KA^1 + (KA^1)^\top ) + L \big) X \\
&= \frac{\alpha \bpr}{\bro} x_1^2 - \Big( \alpha \bro + \frac{\beta \brt \bpt}{\bro \bet}\Big) x_2^2 + \Big( \frac{\beta \bpt}{\bro} - \frac{\gamma \bar{\kappa}}{\tau}\Big) x_3^2 + \Big( \frac{\gamma}{\bro \bet} + \frac{1}{\bar{\kappa} \brt}\Big) x_4^2 + \\
&\; + \frac{1}{\bro} \big( \beta \bpr + \alpha \bpt \big) x_1 x_3 + \frac{1}{\bro \bet} \big( \gamma \brt \bpt - \beta \big) x_2 x_4.
\end{aligned}
\]

Let us choose $\alpha$, $\beta$ and $\gamma$ such that
\[
\alpha = \delta^3 \alpha_0, \quad \beta = - \delta^2 \beta_0, \quad \gamma = - \delta \gamma_0,
\]
where $\alpha_0, \beta_0, \gamma_0 > 0$ and $0 < \delta \ll 1$ are constants to be determined. Then the quadratic form reads
\[
Q(X) = a_1 x_1^2 + a_2 x_2 ^2 + a_3 x_3^2 + a_4 x_4^2 + b_{13} x_1 x_3 + b_{24}x_2 x_4,
\]
where,
\[
\begin{aligned}
a_1 &:= \delta^3 \, \frac{\alpha_0 \bpr}{\bro},\\
a_2 &:= \delta^2 \left( \frac{\beta_0 \brt \bpt}{\bro \bet} - \delta \alpha_0 \bro\right),\\
a_3 &:= \delta \left( \frac{\gamma_0 \bar{\kappa}}{\tau} - \delta \frac{\beta_0 \bpt}{\bro} \right),\\
a_4 &:= \frac{1}{\bar{\kappa} \brt} - \delta \frac{\gamma_0}{\bro \bet},\\
b_{13} &:= \frac{\delta^2}{\bro} \left( \delta \alpha_0 \bpt - \beta_0 \bpr \right),\\
b_{24} &:= \frac{\delta}{\bro \bet} \left( \delta \beta_0 - \gamma_0 \brt \bpt \right).
\end{aligned}
\]

Assuming that
\begin{equation}
\label{condsas}
\begin{aligned}
a_1 &> 0,\\
a_4 &> 0,\\
a_2 - \frac{b_{24}^2}{2 a_4} &> 0,\\
a_3 - \frac{b_{13}^2}{2a_1} &> 0,
\end{aligned}
\end{equation}
clearly we have
\[
Q(X) \geq \tfrac{1}{2}a_1 x_1^2 + \left( a_2 -  \frac{b_{24}^2}{2 a_4} \right) x_2^2 + \left( a_3 - \frac{b_{13}^2}{2a_1} \right) x_3^2 + \tfrac{1}{2} a_4 x_4^2  \geq C |X|^2 > 0,
\]
for all $X \neq 0$, $X \in \R^4$ and some positive constant satisfying,
\[ 
0 < C < \tfrac{1}{2}\min \left\{ \tfrac{1}{2}a_1, \tfrac{1}{2}a_4, a_2 - \frac{b_{24}^2}{2a_4}, a_3 - \frac{b_{13}^2}{2a_1} \right\}.
\]
Therefore, we need to find values of $\alpha_0,\beta_0, \gamma_0 > 0$ and $0 < \delta \ll 1$ sufficiently small such that conditions \eqref{condsas} hold.

First, notice that under assumptions \eqref{H1} - \eqref{H3} and $\alpha_0 > 0$, the first condition in \eqref{condsas} is already satisfied. If we further choose parameter values $\alpha_0$, $\beta_0$ and $\gamma_0$ such that
\begin{equation}
\label{212}
\frac{\gamma_0 \bar{\kappa}}{\tau} - \frac{\beta_0^2 \bpr}{2 \alpha_0 \bro} > 0,
\end{equation}
then, for $\delta > 0$ sufficiently small such that
\begin{equation}
\label{condE}
0 < \delta < \frac{2 \bro \bpr}{\alpha_0 \bpt^2} \left( \frac{\gamma_0 \bar{\kappa}}{\tau} - \frac{\beta_0^2 \bpr}{2 \alpha_0 \bro} \right),
\end{equation}
we can assure that the fourth condition in \eqref{condsas} also holds, as the reader may easily verify. For small $\delta$ we write
\[
\frac{1}{2a_4} = \frac{1}{2} \left( \frac{1}{\bar{\kappa} \brt} - \delta \frac{\gamma_0}{\bro \bet}\right)^{-1} = \frac{1}{2} \bar{\kappa}\brt + \delta \frac{\bar{\kappa}^2 \brt^2 \gamma_0}{2 \bro \bet} + O(\delta^2).
\]
Hence, it suffices to take $\delta$ small enough such that
\begin{equation}
\label{condF}
0 < \delta < \frac{\bro \bet}{\bar{\kappa} \brt \gamma_0},
\end{equation}
and to choose values of $\beta_0$ and $\gamma_0$ satisfying
\begin{equation}
\label{216}
\frac{\brt \bpt}{\bro \bet} \left( \beta_0 - \gamma_0^2 \frac{\bar{\kappa} \brt^2 \bpt}{2 \bro \bet}\right) > 0,
\end{equation}
in order to obtain
\begin{equation}
\label{215}
a_2 - \frac{b_{24}^2}{2 a_4} = \frac{\brt \bpt}{\bro \bet} \left( \beta_0 - \gamma_0^2 \frac{\bar{\kappa} \brt^2 \bpt}{2 \bro \bet}\right) + O(\delta) > 0,
\end{equation}
that is, the third condition in \eqref{condsas}. Finally, the second inequality in \eqref{condsas} follows from \eqref{condF}.

Hence, it suffices to choose positive values of $\alpha_0, \beta_0, \gamma_0$ such that conditions \eqref{216} and \eqref{212} hold. For instance, we can define
\[
\begin{aligned}
\alpha_0 &:= \frac{\tau^2 \brt^2 \bpt^2 \bpr}{\bro^2} > 0,\\
\beta_0 &:= \bpt > 0,\\
\gamma_0 &:= \frac{\bro \bet}{\bar{\kappa} \brt^2} > 0
\end{aligned}
\]
(all positive because of \eqref{H1} - \eqref{H3}). Once these values are determined, we can always find $0 < \delta \ll 1$ sufficiently small such that \eqref{condE}, \eqref{condF} and \eqref{215} hold as well. Substitute $\alpha = \delta^4 \alpha_0$, $\beta = - \delta^2 \beta_0$ and $\gamma = - \delta \gamma_0$ back into the expression of $K$ to obtain the result.
\end{proof}

\section{Linear decay rates}
\label{seclindec}

In this section we describe how to obtain decay rates for solutions to the linearized system \eqref{lind} using the properties of the compensating function $K$. We gloss over some details, because the arguments are very similar to those in the case of hyperbolic conservation laws with relaxation (see section 3 in \cite{KY09}), with a slight modification due to the presence of viscous and relaxation terms combined. It is also to be noticed that we are not applying the equivalence result (Theorem \ref{theoequiv}) inasmuch we are explicitly providing the form of $K$. The estimates hold for both the pure relaxation ($\nu \equiv 0$) and the viscous with thermal relaxation ($\nu > 0$) cases.

Let us denote the standard inner product in $\C^n$ as $\< \, , \, \>$ and let $[A]^s := \tfrac{1}{2} (A + A^\top)$ be the symmetric part of any real matrix $A$. Under the previous assumptions, namely, that
\begin{itemize}
\item[(i)] $A^j$, $L$, $B$, $j = 0,1$, are real symmetric matrices;
\item[(ii)] $A^0 > 0$, $L, B \geq 0$; and
\item[(iii)] there exists a compensating function $K$,
\end{itemize}
let $U$ be the solution to linearized system \eqref{lind}. 
\begin{lemma}
There exists $k > 0$ such that the solutions $U$ to the linear system \eqref{lind} satisfy
\begin{equation}
\label{est12}
|\hU(\xi,t)| \leq C |\hU(\xi,0)| \exp \left( - \frac{k \xi^2 t}{1 + \xi^2}\right),
\end{equation}
for all $t \geq 0$, $\xi \in \R$ and some uniform constant $C > 0$.
\end{lemma}
\begin{proof}
Take the Fourier transform to get equation \eqref{Foulin}. Since the coefficient matrices are symmetric, if we take the inner product of \eqref{Foulin} with $\hU$ and take the real part we obtain
\begin{equation}
\label{la4}
\tfrac{1}{2} \partial_t \< \hU, A^0 \hU \> + \< \hU, L\hU \> + \xi^2 \<\hU, B\hU\> = 0.
\end{equation}

Now multiply \eqref{Foulin} by $-i\xi K$ and take the inner product with $\hU$. The result is
\[
-  \< \hU, i\xi KA^0 \hU_t \> + \xi^2 \< \hU, KA^1 \hU \> - \< \hU, i\xi KL \hU \> - \< \hU, i \xi^3 KB \hU \> = 0.
\]
Use the fact that $KA^0$ is skew-symmetric to verify that 
\[
\Re \< \hU, i\xi KA^0 \hU_t \> = \tfrac{1}{2} \xi \partial_t \< \hU, iKA^0 \hU \>.
\]
Thus, taking the real part of the previous equation yields
\[
- \tfrac{1}{2} \xi \partial_t \< \hU, iKA^0 \hU \> + \xi^2 \< \hU, [KA^1]^s \hU \> = \Re \big( i \xi \< \hU, KL \hU \> \big) + \Re \big( i\xi^3 \< \hU, KB \hU \> \big).
\]
Since $L, B \geq 0$ and by symmetry, we obtain the estimate
\begin{equation}
\label{la7}
- \tfrac{1}{2} \xi \partial_t \< \hU, iKA^0 \hU \> + \xi^2 \< \hU, [KA^1]^s \hU \> \leq \epsilon \xi^2 |\hU|^2 + C_\epsilon \big( \< \hU, L\hU \> + \xi^4 \< \hU, B\hU \> \big),
\end{equation}
for any $\epsilon > 0$ and where $C_\epsilon > 0$ is a uniform constant depending only on $\epsilon > 0$, $|K L^{1/2}|$ and $|KB^{1/2}|$. Now multiply equation \eqref{la4} by $1 + \xi^2$, equation \eqref{la7} by $\delta > 0$ and add them up. The result is
\begin{equation}
\label{la8}
\begin{aligned}
\tfrac{1}{2} \partial_t &\left( (1 + \xi^2) \< \hU, A^0 \hU \> - \delta \xi \< \hU, iKA^0 \hU \> \right) +   \< \hU, L\hU \> + \xi^4 \< \hU, B \hU \> + \\ 
&\; + \xi^2 \left( \delta \< \hU, [KA^1]^s \hU \> + \< \hU, L\hU \> + \< \hU, B \hU \> \right) \\
&\leq \epsilon \delta \xi^2 |\hU|^2 + \delta C_\epsilon \big( \< \hU, L\hU \> + \xi^4 \< \hU, B \hU \>\big).
\end{aligned}
\end{equation}

Now define
\[
M:= \< \hU, A^0 \hU \> - \frac{\delta \xi}{1 + \xi^2} \< \hU, iKA^0 \hU \>.
\]
Notice that $M$ is real because $A^0$ is symmetric and $KA^0$ is skew-symmetric. Since $A^0 > 0$ there exists $C_0 > 0$ such that $\< \hU, A^0 \hU \> \geq C_0|\hU|^2$. It is then easy to show that there exists $\delta_0 > 0$, sufficiently small, such that if $0 < \delta < \delta_0$ then
\[
\frac{1}{C_1}|\hU|^2 \leq M \leq C_1 |\hU|^2,
\]
for some uniform $C_1 > 0$. 

Now from property (b) of the compensating function $K$ (see Definition \ref{defK}), there exists $\gamma > 0$ such that $\< \hU, ([KA^1]^s + L + B) \hU \> \geq \gamma |\hU|^2$. Therefore, by taking $0 <\delta < 1$ we arrive at
\[
\< \hU, (\delta [KA^1]^s + L + B) \hU \> \geq \delta \gamma |\hU|^2.
\]
Choose $\epsilon = \gamma/2$ and $0 < \delta < \min \{ 1, \delta_0, 1/C_\epsilon \}$ to obtain
\[
\tfrac{1}{2} \partial_t M + \tfrac{1}{2} \left( \frac{\xi^2}{1 + \xi^2} \right) \delta \gamma |\hU|^2 + \frac{(1 - \delta C_\epsilon)}{1 + \xi^2} \big( \< \hU, L\hU \> + \xi^4 \< \hU, B \hU \>\big) \leq 0.
\]
This yields
\[
\tfrac{1}{2} \partial_t M + \frac{2k \xi^2}{1 + \xi^2} M \leq 0,
\]
with $k = \tfrac{1}{2} \delta \gamma/C_1 > 0$. This inequality readily implies the desired estimate \eqref{est12}.
\end{proof}

\begin{theorem}[linear decay rates]
\label{thmlindecr}
Under the assumptions (i) - (iii) suppose that $U_0 \in H^s(\R) \cap L^1(\R)$, with $s \geq 2$. Then the solution to the Cauchy problem for linear system \eqref{lind} with $U(x,0) = U_0$ satisfies the decay rate
\begin{equation}
\label{lindec}
\| \partial_x^l U \|_{L^2}^2 \leq C \left( e^{-k t} \| \partial_x^l U_0 \|_{L^2}^2 + (1 + t)^{-(l + 1/2)} \| U_0 \|_{L^1}^2 \right),
\end{equation}
for $0 \leq l \leq s-1$ and some uniform $C > 0$.
\end{theorem}
\begin{proof}
Multiply estimate \eqref{est12} by $\xi^{2l}$ to obtain
\[
\int_\R \xi^{2l} |\hU(\xi,t)|^2 \, d\xi \leq C \int_{\R} \xi^{2l} |\hU(\xi,0)|^2  \exp \left( - \frac{2k \xi^2 t}{1 + \xi^2}\right) \, d\xi =: C(I_1(t) + I_2(t)),
\]
where $I_1$ denotes the integral on the right hand side computed on the set $\xi \in (-1,1)$ and $I_2$ is the integral on $|\xi| > 1$. Since $\xi^2/(1+\xi^2) \geq \tfrac{1}{2}\xi^2$ for $\xi \in (-1,1)$, we have the estimate
\[
\begin{aligned}
I_1(t) &= \int_{-1}^1 \xi^{2l} |\hU(\xi,0)|^2  \exp \left( - \frac{2k \xi^2 t}{1 + \xi^2}\right) \, d\xi \\&\leq \Big( \sup_{\xi \in \R } |\hU_0(\xi)|^2 \Big) \int_{-1}^1 \xi^{2l} e^{-k\xi^2 t} \, d\xi \\ 
&\leq \| U_0 \|_{L^1}^2 \int_{-1}^1 \xi^{2l} e^{-k\xi^2 t} \, d\xi.
\end{aligned}
\]
Using standard calculus tools it is easy to verify that
\[
A(t) := (1 + t)^{l + 1/2} \int_{-1}^1 \xi^{2l} e^{-k\xi^2 t} \, d\xi
\]
is continuous and uniformly bounded above for all $t \geq 0$. Therefore we arrive at
\[
I_1(t) \leq C (1 + t)^{-(l + 1/2)} \| U_0 \|_{L^1}^2,
\]
for some $C > 0$ and all $t \geq 0$. Now, if $\xi^2 \geq 1$ then clearly $\exp( -2k\xi^2 t /(1+\xi^2)) \leq e^{-kt}$. Together with Plancherel's theorem, this yields the estimate
\[
\begin{aligned}
I_2(t) &= \int_{|\xi| \geq 1} \xi^{2l} |\hU(\xi,0)|^2  \exp \left( - \frac{2k \xi^2 t}{1 + \xi^2}\right) \, d\xi \\&\leq e^{-kt} \int_\R \xi^{2l} |\hU_0(\xi)|^2 \, d\xi \\
&\leq e^{-kt} \|\partial_x^l U_0 \|_{L^2}^2.
\end{aligned}
\]
Combining both estimates we arrive at \eqref{lindec}.
\end{proof}

\begin{corollary}
\label{corCC}
Under hypotheses \eqref{H1} - \eqref{H3} for a compressible fluid, let $\bU = (\bro, \bru, \brt, 0)^\top \in \mathcal{V}$ be a constant equilibrium state. If $U_0 - \bU \in H^s(\R) \cap L^1(\R)$, with $s \geq 2$, is an initial perturbation (with finite energy and finite mass) of the equilibrium state $\bU$ then the solutions $U - \bU$ to the linearized equations around $\bU$ satisfy the decay estimates
\begin{equation}
\label{lindecU}
\| \partial_x^l (U - \bU) \|_{L^2}^2 \leq C \left( e^{-k t} \| \partial_x^l (U_0 - \bU) \|_{L^2}^2 + (1 + t)^{-(l + 1/2)} \| U_0 - \bU \|_{L^1}^2 \right),
\end{equation}
for $0 \leq l \leq s-1$ and some uniform $C, k > 0$. These linear decay rates hold for solutions to the linearization of both the viscous Cattaneo-Christov system \eqref{CCvisc} (for which $\nu > 0$) and the inviscid Cattaneo-Christov model \eqref{CCrel} (for which $\nu \equiv 0$).
\end{corollary}
\begin{proof}
Both systems \eqref{CCvisc} and \eqref{CCrel} can be recast in the quasilinear symmetric form \eqref{nlsyst}, where the solutions are written as $U - \bU$, that is, as perturbations of the equilibrium state. Under hypotheses \eqref{H1} - \eqref{H3}, the coefficients $A^0$, $A^1$, $B$ and $L$ satisfy assumptions (i) - (iii), where the compensating function $K$ is given by \eqref{viscK} in the viscous case ($\overline{\nu} > 0$), and by \eqref{relaxK} in the pure thermal relaxation case ($\overline{\nu} \equiv 0$). Thus, the hypotheses of Theorem \ref{thmlindecr} are satisfied and any solution $U - \bU$ to the linearized system \eqref{lind} with initial condition $U_0 - \bU$ obeys the desired linear decay rates, as claimed.
\end{proof}

\section{Discussion}

In this paper we have shown that one-dimensional Cattaneo-Christov systems for compressible fluid flow are strictly dissipative. This property holds for the case in which viscous and thermal relaxation effects are combined, as well as for the case where viscosity is neglected and the only dissipation terms are due to thermal relaxation. We have proved strict dissipativity for these systems by verifying the genuine coupling condition, as well as by providing explicit forms for the compensating functions which allow, in turn, to establish energy estimates leading to the decay structure of solutions to the linearized problem around equilibrium states. 

In the process, we have shown, for instance, that Cattaneo-Christov systems in one dimension are symmetrizable. As we have pointed out, symmetrizability is a fundamental property in the theory. It is natural to ask whether multi-dimensional Cattaneo-Christov systems are strictly dissipative. With respect to this problem, it is important to remark, however, that not even the existence of a symmetrizer in several space dimensions is yet clear. As the seasoned reader might readily have noticed, the material derivative in Christov's constitutive law in more than one dimension prevents expression \eqref{symm} to be a symmetrizer for the multi-dimensional case. This is the subject of current investigations. 

Finally, even though the estimates performed to obtain the decay rates in Theorem \ref{thmlindecr} are very similar (at the linear level) to those for hyperbolic balance laws \cite{KY09} (see also \cite{KaTh83}), we call upon the attention of the reader that the statement of Corollary \ref{corCC} should not be taken for granted. For instance, the analyses pertaining to the local existence of solutions for viscous systems of conservation laws \cite{KaTh83,Ser10b}, the global existence of solutions for hyperbolic balance laws \cite{HaNa03,Y04}, as well as the global stability of constant equilibrium states for dissipative balance laws \cite{RuSe04}, they all consider the existence of a convex entropy structure which lacks in the present case because the system is not in conservation form. Therefore, the linear decay rates around equilibrium states for Cattaneo-Christov systems constitute the first step to show that constant equilibrium states are asymptotically stable under small perturbations even in the absence of a convex entropy.

\section*{Acknowledgements}

RGP thanks Jeffrey Humpherys for useful discussions. The work of FA was partially supported by CONACyT (Mexico), through a scholarship for doctoral studies, grant no. 465484. The work of RGP was partially supported by DGAPA-UNAM, program PAPIIT, grant IN-100318.

\def\cprime{$'$}





\end{document}